\documentclass[letterpaper]{amsart}
\usepackage[headings]{fullpage}
\usepackage{amsmath}
\usepackage{amssymb}
\usepackage{hyperref}
\usepackage{enumitem}
\numberwithin{equation}{section} 
\usepackage{xcolor}

\newtheorem{theorem}{Theorem}[section]
\newtheorem{lemma}[theorem]{Lemma}

\newtheorem{proposition}[theorem]{Proposition}

\theoremstyle{definition}
\newtheorem{definition}[theorem]{Definition}

\newcommand\R{\mathbb{R}}
\newcommand\Z{\mathbb{Z}}

\newcommand{\qtq}[1]{\quad\text{#1}\quad}

\renewcommand{\epsilon}{\varepsilon}

\newcommand{\jb}[1]{\langle#1\rangle}

\newcommand{\rn}[1]{\textup{\uppercase\expandafter{\romannumeral#1}}}

\allowdisplaybreaks

\begin{document}

\title{Dispersive decay for the energy-critical nonlinear wave equation}

\author{Matthew Kowalski}
\address{Department of Mathematics, University of California, Los Angeles, CA 90095, USA}
\email{mattkowalski@math.ucla.edu}

\begin{abstract}
    We prove pointwise-in-time dispersive decay for solutions to the energy-critical nonlinear wave equation in spatial dimension $d = 3$.
\end{abstract}

\maketitle

\section{Introduction}\label{intro}
    \noindent In this paper, we demonstrate dispersive decay for the {\em energy-critical nonlinear wave equation}:
\begin{equation}\label{NLW}\tag{NLW}\begin{cases} 
    \partial_t^2u - \Delta u \pm u^5 = 0 \\ 
    u(0,x) = u_0(x) \in \dot{H}^1(\R^3) \\
    \partial_t u(0,x) = u_1(x) \in L^2(\R^3),
\end{cases}\end{equation}
where $u(t,x)$ is a real-valued function on spacetime $\R_t\times \R^3_x$, $+$ represents the defocusing equation, and $-$ represents the focusing equation.

For dispersive models, our understanding of the long-time behavior of solutions is naturally tied to {\em dispersion}: the property that different frequencies propagate at different velocities.
In classical works, dispersion manifested primarily through pointwise-in-time dispersive decay; see e.g.~\cite{cubic-1986, we-1985, general-1983, cubic-1978, segal}. For the linear half-wave equation, this is exemplified by the dispersive estimate
\begin{equation}\label{intro/non-localized linear dispersive decay}
        \|e^{\pm i t |\nabla|} f \|_{L^p} \lesssim |t|^{-(d-1)(\frac{1}{2} - \frac{1}{p})} \||\nabla|^{(d+1)(\frac{1}{2} - \frac{1}{p})}f\|_{L^{p'}} \quad \text{for } t \neq 0 \text{ and } 2 \leq p < \infty.
\end{equation}
For nonlinear models, consistent with the well-posedness theory at that time, these classical dispersive decay estimates required high-regularity initial data.
In contemporary works, dispersion has shifted to appear through Strichartz estimates and the resulting spacetime bounds; see e.g.~\cite{eNLS-wp-d3,focusing-GWP,spacetime-bounds-tao} as well as Theorems \ref{sb/strichartz} and \ref{sb/sb}.

Given recent successes in scaling-critical well-posedness and scattering (see e.g.~Theorem \ref{well-posedness}), a string of authors \cite{mNLS-dd, cNLS-dd-constant, eNLS-dd-H3, eNLS-dd, eNLS-dd-H1+, mgKdV-dd} has returned to these pointwise-in-time estimates and sought to lower the regularity required for dispersive decay.
In \cite{mNLS-dd,eNLS-dd}, these efforts successfully demonstrated dispersive decay for the mass-critical and energy-critical nonlinear Schr\"odinger equations under minimal hypotheses.
Subsequent progress has focused on the mass-critical gKdV equation in \cite{mgKdV-dd}.
For a full history of dispersive decay, we direct the interested reader to the aforementioned string of authors.

The (non)linear wave equation has a rich history of decay estimates which we do not attempt to fully recount here. For decay estimates of \eqref{NLW} on nonstationary spacetimes under stronger regularity and decay hypotheses, we direct the reader to the work of Shi-Zhuo Looi in \cite{looi-2, looi-1} and references therein.
For work on dispersive estimates for various wave equations, we direct the interested reader to \cite{we-1980,we-1985,general-1983,looi-2,rel-we-1972, segal} and references therein.

Before stating our main theorem, let us briefly recall the well-posedness and scattering theory for this model. 
The initial-value problem \eqref{NLW} is {\em energy-critical}: the scaling symmetry, $u(t,x) \mapsto \lambda^\frac{1}{2} u(\lambda t, \lambda x)$, preserves the (conserved) energy
\begin{equation*}
    E(u,\partial_t u) = \tfrac{1}{2} \int_{\R^3} |\nabla u|^2 + |\partial_t u|^2 - \tfrac{1}{3} |u|^6 dx,
\end{equation*}
as well as the $\dot{H}^1 \times L^2$ norm. 
Global well-posedness in the energy-space $\dot{H}^1 \times L^2$ was established for the defocusing equation by a culmination of many authors; see \cite{defocusing-Bahouri, defocusing-Grillakis,defocusing-Grillakis-2,defocusing-Sha-Struwe,defocusing-Sha-Struwe-2,defocusing-Struwe} and references therein. For the focusing equation, global well-posedness was established by Kenig and Merle in \cite{focusing-GWP}.
For a history of well-posedness up to their respective times, see \cite{shatah-struwe, sogge}.
For a recent installment, see \cite{spacetime-bounds-tao}.

We recall the necessary well-posedness results in the following theorem \cite{defocusing-Bahouri, defocusing-Grillakis, defocusing-Grillakis-2, focusing-GWP, defocusing-Sha-Struwe, defocusing-Sha-Struwe-2,defocusing-Struwe}:
\begin{theorem}\label{well-posedness}
    Let there exist $(u_0, u_1) \in \dot{H}^1 \times L^2(\R^3)$. In the focusing case, suppose that $(u_0,u_1)$ satisfies
    \begin{equation*}
        E(u_0,u_1) < E(W,0) \qtq{and} \|u_0\|_{\dot{H}^1} < \|W\|_{\dot{H}^1},
    \end{equation*}
    where
    $W(x) = \big( 1 + \tfrac{1}{3} |x|^2\big)^{-\frac{1}{2}}$
    is a stationary solution to \eqref{NLW}.
    Then there exists a unique global solution $u \in C_t\dot{H}^1_x(\R_t \times \R_x^3)$ to \eqref{NLW} with initial data $(u_0,u_1)$ which satisfies
    \begin{equation}\label{intro/bound}
        \|u\|_{L_{t,x}^8(\R \times \R^3)} \leq C(E(u_0,u_1)).
    \end{equation}
    Moreover, there exist scattering states $(u_0^{\pm},u_1^{\pm}) \in \dot{H}^1 \times L^2$ such that
    \begin{equation}\label{intro/scattering}
        \lim_{t \to \pm \infty} \Big\|u(t) - \cos(t|\nabla|) u_0^\pm - \tfrac{\sin(t|\nabla|)}{|\nabla|}u_1^\pm\Big\|_{\dot{H}^1} = 0.
    \end{equation}
\end{theorem}

In the theorem to follow, we show dispersive decay analogous to \eqref{intro/non-localized linear dispersive decay} for solutions to \eqref{NLW} in the energy-space $\dot{H}^1\times L^2$. We believe that this theorem is the optimal nonlinear analogue of the dispersive decay \eqref{intro/non-localized linear dispersive decay}. Notably, it requires only that initial data lie in the energy-space $\dot{H}^1 \times L^2$, it exhibits a linear dependence on the initial data, and it has constants that depend only on the size of the initial data.
\begin{theorem}\label{theorem}
    Fix $2 < p < \infty$. Given $(u_0,u_1) \in \big(\dot{H}^{2 - \frac{4}{p},p'} \times \dot{H}^{1 - \frac{4}{p},p'}\big) \cap (\dot{H}^1 \times L^2)(\R^3)$ which satisfies the hypotheses of Theorem \ref{well-posedness}, let $u(t)$ denote the corresponding global solution to \eqref{NLW}. Then
    \begin{equation}\label{theorem/decay}
        \|u(t)\|_{L^p} \lesssim C\big(\|(u_0,u_1)\|_{\dot{H}^1 \times L^2},p\big)|t|^{-(1 - \frac{2}{p})} \Big[\big\||\nabla|^{2 - \frac{4}{p}}u_0\big\|_{L^{p'}} + \big\||\nabla|^{1 - \frac{4}{p}}u_1\big\|_{L^{p'}}\Big],
    \end{equation}
    uniformly for $t \neq 0$. 
    A stronger, Besov version of \eqref{theorem/decay} is shown within the proof; see \eqref{proof/stronger}.
\end{theorem}
Absent from Theorem \ref{theorem} is decay in $L^\infty$. This restriction arises from the high number of derivatives present in the right-hand side of \eqref{intro/linear dispersive decay} which prevents the use of Schur's test to recombine frequencies; see \eqref{proof/infty failure 1} and the surrounding discussion. 
If we restrict our attention to initial data in the stronger Besov space $\dot{B}^2_{1,1} \times \dot{B}^1_{1,1}$ (see Definition \ref{intro/besov}), we recover full dispersive decay in $L^\infty$:
\begin{theorem}\label{edge theorem}
    Given $(u_0,u_1) \in (\dot{B}^2_{1,1} \times \dot{B}^1_{1,1}) \cap (\dot{B}^1_{2,1} \times \dot{B}^0_{2,1})$ which satisfies the hypotheses of Theorem \ref{well-posedness}, let $u(t)$ denote the corresponding global solution to \eqref{NLW}. Then
    \begin{equation}\label{edge theorem/decay}
        \|u(t)\|_{L_x^\infty} \leq C\big(\|(u_0,u_1)\|_{\dot{B}^1_{2,1} \times \dot{B}^0_{2,1}}\big)|t|^{-1} \Big[\big\|u_0\big\|_{\dot{B}^2_{1,1}} +\big\|u_1\big\|_{\dot{B}^1_{1,1}}\Big],
    \end{equation}
    uniformly for $t \neq 0$. A stronger version of \eqref{edge theorem/decay} is shown within the proof; see \eqref{edge/stronger}.
\end{theorem}

Our interest in this particular model is to expand the methods presented in \cite{mNLS-dd,eNLS-dd} and inspire continued work in this area.
In contrast to \cite{mNLS-dd,eNLS-dd}, whose arguments are based on spacetime bounds in Lorentz spaces, we instead rely on frequency-localized estimates and careful summation.
These new methods are necessary to address the loss of derivatives in \eqref{intro/non-localized linear dispersive decay}.

It is well known that the linear half-wave equation, and hence the linear wave equation, enjoys a dispersive decay in $L^p$ for $2 \leq p \leq \infty$ of the form
\begin{equation}\label{intro/linear dispersive decay}
    \big\| e^{\pm i t |\nabla|} P_N f\big\|_{L^p} \lesssim N^{d(1 - \frac{2}{p})} (1 + |t|N)^{-(d-1)(\frac{1}{2} - \frac{1}{p})} \|P_N f\|_{L^{p'}},
\end{equation}
uniformly for $t \in \R$ and $N \in 2^\Z$ where $P_N$ is a Littlewood--Paley projection (see \eqref{intro/Littlewood--Paley}).
Due to the nonlinearity, frequency-localized decay of the form \eqref{intro/linear dispersive decay} is not possible for solutions to \eqref{NLW}; indeed, this fails for any initial data with finite Fourier support. Instead, from the boundedness of the Littlewood--Paley square function \eqref{intro/square function} we immediately gain the linear dispersive decay \eqref{intro/non-localized linear dispersive decay}, of which Theorem \ref{theorem} is a nonlinear analogue.



Absent from \eqref{intro/non-localized linear dispersive decay}, and correspondingly from Theorem \ref{theorem}, is dispersive decay in $L^\infty$. Though present in \eqref{intro/linear dispersive decay}, square function estimates are insufficient to recover $L^\infty$ decay of the form \eqref{intro/non-localized linear dispersive decay}. Instead, \eqref{intro/linear dispersive decay} and countable subadditivity imply the weaker result:
\begin{equation}\label{intro/Linfty decay}
    \|e^{\pm i t |\nabla|} f \|_{L^\infty} \lesssim \|e^{\pm i t |\nabla|} f\|_{\dot{B}^0_{\infty,1}} \lesssim |t|^{-\frac{d-1}{2}} \|f\|_{\dot{B}^\frac{d+1}{2}_{1,1}},
\end{equation}
of which Theorem \ref{edge theorem} is the nonlinear analogue.

    

    \subsection*{Acknowledgements}
    The author was supported in part by NSF grants DMS-2154022 and DMS-2054194.
    The author is grateful to Monica Vi\c{s}an and Rowan Killip for their continued guidance and confidence.

    \subsection*{Notation and Preliminaries}

We use the standard notation $A \lesssim B$ to indicate that $A \leq C B$ for some universal constant $C > 0$ that may change from line to line. If both $A \lesssim B$ and $B \lesssim A$ then we use the notation $A \sim B$. When the implied constant fails to be universal, the relevant dependencies will be indicated within the text or included as subscripts on the symbol.




We define the {\em Littlewood--Paley projections} as follows: Let $\varphi$ denote a smooth bump function supported on $\{|\xi| \leq 2\}$ such that $\varphi(\xi) = 1$ for $|\xi| \leq 1$. For dyadic numbers $N \in 2^\Z$, we then define $P_{\leq N}$, $P_{> N}$, and $P_N$ as
\begin{equation}\begin{split}\label{intro/Littlewood--Paley}
    \widehat{P_{\leq N} f} & = \varphi(\xi/N) \widehat{f}(\xi) \\
    \widehat{P_{> N} f} & = [1-\varphi(\xi/N)] \widehat{f}(\xi) \\
    \widehat{P_N f}(\xi) & = [\varphi(\xi/N) - \varphi(2\xi/N)] \widehat{f}(\xi).
\end{split}\end{equation}
We will often denote $P_N f = f_N$, $P_{\leq N}f = f_{\leq N}$ and $P_{>N}f = f_{>N}$. When a function $f(t,x)$ depends on both time and space, we let $f_N(t,x)$ denote the Littlewood--Paley projection of $f$ in only the spatial variable $x$.

As Fourier multipliers, the Littlewood--Paley projections commute with derivative operators and the half-wave propagator $e^{\pm i t |\nabla|}$. Moreover, they are bounded on $L^p$ for all $1 \leq p \leq \infty$ and on $\dot{H}^s$ for all $s \in \R$. For $1 < p < \infty$, we have that
\begin{equation*}
    \sum_{N \in 2^\Z} P_N f \to f \quad\text{in $L^p$}.
\end{equation*}
In addition, $P_N$, $P_{\leq N}$, and $P_{>N}$ are bounded pointwise by a constant multiple of the Hardy--Littlewood maximal function,
\begin{equation*}
    |P_N f| + |P_{\leq N} f| \lesssim Mf,
\end{equation*}
uniformly for $N \in 2^\Z$.

Associated to the Littlewood--Paley projections are the {\em Bernstein inequalities}, which state
\begin{equation}\begin{split}\label{intro/bernstein}
    \||\nabla|^s P_N f\|_{L^p} & \sim N^s \|P_N f\|_{L^p} \\
    \||\nabla|^s P_{\leq N} f\|_{L^p} & \lesssim N^s \|P_N f\|_{L^p} \\
    \|P_N f\|_{L^p},\|P_{\leq N} f\|_{L^p} & \lesssim N^{d(\frac{1}{q} - \frac{1}{p})} \|P_N f\|_{L^q}
\end{split}\end{equation}
for all $s \geq 0$, $1 \leq q \leq p \leq \infty$ and uniformly for $N \in 2^\Z$.
Additionally, we have the Littlewood--Paley square function estimate which states
\begin{equation}\label{intro/square function}
    \big\|\|f_N\|_{\ell_N^2}\big\|_{L^p} \sim \|f\|_{L^p},
\end{equation}
for $1 < p < \infty$.

For Theorem \ref{edge theorem}, it will be necessary to consider initial data in homogeneous Besov spaces. We recall the definition here:
\begin{definition}\label{intro/besov}
    Fix $d \geq 1$, $1 \leq p,q \leq \infty$, and $s \in \R$. The homogeneous Besov space $\dot{B}^s_{p,q}$ is the completion of the Schwartz functions $\mathcal{S}(\R^d)$ with respect to the norm
    \begin{equation*}
        \|f\|_{\dot{B}^s_{p,q}} = \big\| N^s\|f_N(x)\|_{L^p_x(\R^d)}\big\|_{\ell_N^q(2^\Z)}.
    \end{equation*}
    We note that the Bernstein inequalities \eqref{intro/bernstein} imply that the factor $N^s$ acts as $|\nabla|^s$.
\end{definition}

For our Besov analysis, it will be necessary to consider the Littlewood--Paley projection of the product of functions. We recall two simple paraproduct decompositions in the following lemma:
\begin{lemma}[Paraproduct decompositions]\label{intro/paraproduct}
    Suppose that $f \in L^2$. Then
    \begin{align}
        |P_N(f^5)| 
        & \lesssim \sum_{\substack{N_1 \sim N_2 \gtrsim N \\ N_5 \leq \dots \leq N_1}}M\big[f_{N_1}f_{N_2}f_{N_3}f_{N_4}f_{N_5}\big] + \sum_{\substack{N_1 \sim N \\ N_5 \leq \dots \leq N_1}}M\big[f_{N_1}f_{N_2}f_{N_3}f_{N_4}f_{N_5}\big] \label{intro/paraproduct/full}\\
        & \lesssim \sum_{\substack{N_1 \gtrsim N \\ N_5 \leq \dots \leq N_1}}M\big[f_{N_1}f_{N_2}f_{N_3}f_{N_4}f_{N_5}\big], \label{intro/paraproduct/simple}
    \end{align}
    uniformly for $N \in 2^\Z$, where $M$ is the Hardy--Littlewood maximal function.
\end{lemma}

We use $L_t^p L_x^q(T \times X)$ to denote the mixed Lebesgue spacetime norm
\begin{equation*}
    \|f\|_{L_t^p L_x^q(T\times X)} = \big\| \|f(t,x)\|_{L^q(X,dx)} \big\|_{L^p(T,dt)} = \bigg[ \int_T \bigg(\int_X |f(t,x)|^q dx\bigg)^{p/q} dt\bigg]^{1/p}.
\end{equation*}
When $p = q$, we let $L^p_{t,x} = L_t^p L_x^p$. When $X = \R^d$, we let $L_t^pL_x^q(T) = L_t^pL_x^q(T\times\R^d)$. 
Similar to this notation, we use $\ell_N^r L_t^p L_x^q(I \times T \times X)$ to denote the mixed Besov spacetime norm
\begin{equation*}
    \|f\|_{\ell_N^r L_t^p L_x^q(I \times T \times X)} = \Big\| \big\| \| f_N(t,x)\|_{L_x^q(X,dx)} \big\|_{L_t^p(T,dt)}\Big\|_{\ell_N^q(I)}.
\end{equation*}
When $I = 2^\Z$ and $X = \R^d$, we let $\ell_N^r L_t^p L_x^q(T) = \ell_N^r L_t^p L_x^q(2^\Z \times T \times \R^d)$. We note that other sources may use the notation $\tilde{L}_t^p \dot{B}^0_{q,r} = \ell_N^r L_t^p L_x^q$; see e.g.~\cite{besov-notation}.

Motivated by the linear dispersive decay \eqref{intro/linear dispersive decay}, we define $\jb{x} = 1 + |x|$.
With this, the linear dispersive decay \eqref{intro/linear dispersive decay} will be often rewritten as
\begin{equation*}
    \big\| e^{\pm i t |\nabla|} P_N f\big\|_{L^p} \lesssim N^{d(1 - \frac{2}{p})} \jb{tN}^{-(d-1)(\frac{1}{2} - \frac{1}{p})} \|P_N f\|_{L^{p'}}.
\end{equation*}
We note that $\jb{x}^{-\alpha}$ is integrable over $\R^d$ for all $\alpha > d$.

\section{Spacetime Bounds}\label{sb}
    Central to our analysis will be estimates of solutions to \eqref{NLW} in various mixed spacetime norms. In the usual way, for a range of admissible exponents, control over these mixed spacetime norms will follow from the bound \eqref{intro/bound} and Strichartz estimates. 

We begin by recalling the admissible exponents and Strichartz estimates.
\begin{definition}[Wave-admissible]
    We say that a pair $(p,q) \in [2,\infty]\times[2,\infty)$ is {\em wave-admissible} if 
    \begin{align*}
        \tfrac{1}{p} + \tfrac{1}{q} \leq \tfrac{1}{2}.
    \end{align*}
\end{definition}
\begin{theorem}[Strichartz estimates]\label{sb/strichartz}
    Suppose that $p,q$ is a wave-admissible pair such that $\frac{1}{p} + \frac{3}{q} = \frac{1}{2} + \gamma$ for some $\gamma$. Then for all spacetime slabs $J \times \R^3$, the half-wave map satisfies
    \begin{equation*}
        \| e^{\pm i t |\nabla|} |\nabla|^{\gamma}f\|_{L_t^p L_x^q(J)} \lesssim \|f\|_{\dot{H}^1}.
    \end{equation*}
    Moreover, for all wave-admissible pairs $(\tilde{p},\tilde{q})$ with $\frac{1}{\tilde{p}} + \frac{3}{\tilde{q}} = \frac{1}{2} + \tilde{\gamma}$ for some $\tilde{\gamma}$,
    \begin{equation*}
        \bigg\| \int_0^t e^{-(t-s)\partial_x^3} |\nabla|^{\gamma} F(s) ds \bigg\|_{L_t^pL_x^q(J)} \lesssim \big\||\nabla|^{-\tilde{\gamma}} F\big\|_{L_t^{\tilde{p}'}L_x^{\tilde{q}'}}.
    \end{equation*}
\end{theorem}

From the Strichartz estimates, a standard bootstrap argument extends the bound \eqref{intro/bound} to all wave-admissible pairs with a suitable number of derivatives. As the proof is standard, we do not recreate it here; instead, we direct the reader to \cite{eNLS-wp-d3,eNLS-wp-d4} for examples of the argument applied to the nonlinear Schr\"odinger equation. For the Besov improvement, we similarly direct the reader to \cite{eNLS-dd}.

\begin{proposition}[Mixed spacetime bounds]\label{sb/sb}
    Given $(u_0, u_1) \in \dot{H}^1 \times L^2$ which satisfies the hypotheses of Theorem \ref{well-posedness}, let $u(t)$ denote the corresponding global solution to \eqref{NLW}. Suppose that $p,q$ is a wave-admissible pair such that $\frac{1}{p} + \frac{3}{q} = \frac{1}{2} + \gamma$ for some $\gamma$. Then
    \begin{equation*}
        \big\||\nabla|^{\gamma} u\big\|_{L_t^p L_x^q} \lesssim \big\||\nabla|^{\gamma} u\big\|_{\ell^2_NL_t^p L_x^q} \leq C\big(\|(u_0,u_1)\|_{\dot{H}^1 \times L^2}\big).
    \end{equation*}
    Moreover, if $(u_0, u_1) \in \dot{B}^1_{2,1} \times \dot{B}^0_{2,1}$, then
    \begin{equation*}
        \big\||\nabla|^{\gamma} u\big\|_{\ell^1_N L_t^p L_x^q} \leq C\big(\|(u_0,u_1)\|_{\dot{B}^1_{2,1} \times \dot{B}^0_{2,1}}\big).
    \end{equation*}
\end{proposition}

\section{Proof of Theorem \ref{theorem}}\label{proof}
    We now prove Theorem \ref{theorem}. This proof will follow the overall bootstrap structure of \cite{mNLS-dd,eNLS-dd}. However, in contrast to \cite{mNLS-dd, eNLS-dd}, we do not use Lorentz spaces to control the decay terms; instead, we use the frequency-localized dispersive decay \eqref{intro/linear dispersive decay} and careful summation.

\begin{proof}[Proof of Theorem \ref{theorem}]
    It suffices to work with $t > 0$ as $t < 0$ will follow from time-reversal symmetry. By the density of Schwartz functions in $(\dot{H}^{2 - \frac{4}{p},p'} \times \dot{H}^{1 - \frac{4}{p},p'}) \cap (\dot{H}^1 \times L^2)$, it suffices to consider Schwartz solutions of \eqref{NLW}.

    To take advantage of the additional decay in \eqref{intro/linear dispersive decay}, we seek to prove the stronger decay result
    \begin{equation}\label{proof/stronger}
        \Big\|\big(\tfrac{\jb{tN}}{N}\big)^{1-\frac{2}{p}}u_N\Big\|_{\ell_N^2 L_t^\infty L_x^p([0,\infty))} \leq C\big(\|(u_0,u_1)\|_{\dot{H}^1 \times L^2}\big)\Big[\big\||\nabla|^{2 - \frac{4}{p}}u_0\big\|_{L^{p'}} + \big\||\nabla|^{1 - \frac{4}{p}}u_1\big\|_{L^{p'}}\Big].
    \end{equation}
    For $0 < T \leq \infty$, we define the norm
    \begin{equation*}
        \|u\|_{X(T)} = \Big\|\big(\tfrac{\jb{tN}}{N}\big)^{1-\frac{2}{p}}u_N\Big\|_{\ell_N^2 L_t^\infty L_x^p([0,T))}.
    \end{equation*}
    It then suffices to show
    \begin{equation}\label{proof/conclusion}
        \|u\|_{X(\infty)} \leq C\big(\|(u_0,u_1)\|_{\dot{H}^1 \times L^2}\big)\Big[\big\||\nabla|^{2 - \frac{4}{p}}u_0\big\|_{L^{p'}} + \big\||\nabla|^{1 - \frac{4}{p}}u_1\big\|_{L^{p'}}\Big],
    \end{equation}
    for which we proceed with a bootstrap argument.

    Let $\eta > 0$ denote a small parameter to be chosen later depending only on universal constants. A quick calculation shows that 
    \begin{equation*}
    (q,r) = \Big(\tfrac{3p^2}{3p-2},\tfrac{18p^2}{3p^2-6p+4}\Big)
    \end{equation*}
    is a wave-admissible pair with $\gamma = \frac{1}{p} + \frac{3}{q} - \frac{1}{2} = 0$. Proposition \ref{sb/sb} then implies that we may decompose $[0,\infty)$ into $J = J(\|(u_0,u_1)\|_{\dot{H^1}\times L^2},\eta)$ many intervals $I_j = [T_{j-1},T_j)$ on which
    \begin{equation}\label{proof/smallness}
        \big\| u\big\|_{\ell_N^2 L_s^q L_x^r(I_j)} < \eta.
    \end{equation}
    
    For all $j = 1,\dots, J$, we aim to show the bootstrap statement:
    \begin{equation}\begin{split}\label{proof/bootstrap}
        \|u\|_{X(T_j)} & \lesssim \big\||\nabla|^{2 - \frac{4}{p}}u_0\big\|_{L^{p'}} + \big\||\nabla|^{1 - \frac{4}{p}}u_1\big\|_{L^{p'}} + C(\|(u_0,u_1)\|_{\dot{H^1}\times L^2})\|u\|_{X(T_{j-1})} + \eta^4 \|u\|_{X(T_j)}.
    \end{split}\end{equation}
    Choosing $\eta > 0$ sufficiently small based on the constants in \eqref{proof/bootstrap}, we could then iterate over the norms $X(T_j)$ for $j = 1,\dots, J(\|(u_0,u_1)\|_{\dot{H^1}\times L^2})$ to yield \eqref{proof/conclusion} and conclude the proof of Theorem \ref{theorem}.

    We therefore focus our attention on \eqref{proof/bootstrap}. Recall the Duhamel formula:
    \begin{equation}\begin{split}\label{proof/duhamel}
        u(t) & = \cos(t|\nabla|) u_0 + \tfrac{\sin(t|\nabla|)}{|\nabla|} u_1 \mp \int_0^t \tfrac{\sin((t-s)|\nabla|)}{|\nabla|} u^5(s) ds \\
        & = \cos(t|\nabla|) u_0 + \tfrac{\sin(t|\nabla|)}{|\nabla|} u_1 + \text{NL}.
    \end{split}\end{equation}
    By the linear dispersive decay \eqref{intro/linear dispersive decay} and the Littlewood-Paley square function estimate, the contribution of the linear term to $\|u(t)\|_{X(T_j)}$ is immediately seen to be acceptable:
    \begin{equation}\label{proof/linear}
        \big\|\cos(t|\nabla|) u_0 + \tfrac{\sin(t|\nabla|)}{|\nabla|} u_1\big\|_{X(T_j)} \lesssim \big\||\nabla|^{2 - \frac{4}{p}}u_0\big\|_{L^{p'}} + \big\||\nabla|^{1 - \frac{4}{p}}u_1\big\|_{L^{p'}}.
    \end{equation}
    We thus focus on the nonlinear correction $\text{NL}$.

    By the linear dispersive decay \eqref{intro/linear dispersive decay} and the Bernstein inequality \eqref{intro/bernstein}, we may estimate
    \begin{equation*}\begin{split}
        \big\|\text{NL}\|_{X(T_j)}
        & \lesssim \bigg\|N^{1 - \frac{4}{p}} \sup_{t \in [0,T_j)} \jb{tN}^{1-\frac{2}{p}} \int_0^t \jb{(t-s)N}^{-(1-\frac{2}{p})} \big\|P_N u^5(s)\big\|_{L_x^{p'}}ds \bigg\|_{\ell_N^2} \\
        & \lesssim \bigg\|N^{1 - \frac{1}{p}}\sup_{t \in [0,T_j)}\jb{tN}^{1-\frac{2}{p}} \int_0^t \jb{(t-s)N}^{-(1-\frac{2}{p})} \big\|P_N u^5(s)\big\|_{L_x^1}ds \bigg\|_{\ell_N^2}.
    \end{split}\end{equation*}
    The $\ell^1_N \hookrightarrow \ell^2_N$ embedding and the simple paraproduct decomposition \eqref{intro/paraproduct/simple} then imply that
    \begin{align*}
        \big\|\text{NL}\|_{X(T_j)} & \lesssim \sum_{\substack{N_1\gtrsim N \\ N_5 \leq \dots \leq N_1}}N^{1 - \frac{1}{p}} \sup_{t \in [0,T_j)} \jb{tN}^{1-\frac{2}{p}} \int_0^t \jb{(t-s)N}^{-(1-\frac{2}{p})} \big\|u_{N_1} u_{N_2} u_{N_3} u_{N_4} u_{N_5}\big\|_{L_x^1} ds.
    \end{align*}
    
    We now decompose $[0,t)$ into $[0,t/2) \cup [t/2,t)$. For $s \in [0,t/2)$, we note that $|t-s| \sim |t|$ and for $s \in [t/2,t)$, we note that $|s| \sim |t|$. We may then estimate
    \begin{align*}
        \big\|\text{NL}\|_{X(T_j)} & \lesssim \sum_{\substack{N_1\gtrsim N \\ N_5 \leq \dots \leq N_1}}N^{1 - \frac{1}{p}} \sup_{t \in [0,T_j)} \int_0^{t/2}
        \big\|u_{N_1} u_{N_2} u_{N_3} u_{N_4} u_{N_5}\big\|_{L_x^1} ds \\
        & \hspace{11pt} + \sum_{\substack{N_1\gtrsim N \\ N_5 \leq \dots \leq N_1}} N^{1 - \frac{1}{p}} \sup_{t \in [0,T_j)}\int_{t/2}^t\jb{(t-s)N}^{-(1-\frac{2}{p})} \jb{sN}^{1- \frac{2}{p}}\big\|u_{N_1} u_{N_2} u_{N_3} u_{N_4} u_{N_5}\big\|_{L_x^1} ds.
    \end{align*}
    Applying H\"older's inequality and taking the supremum over $t \in [0,T_j)$, direct calculation then implies that
    \begin{align*}
        \big\|\text{NL}\|_{X(T_j)} & \lesssim \sum_{\substack{N_1\gtrsim N \\ N_5 \leq \dots \leq N_1}} N^{1-\frac{1}{p}} \big\|\jb{sN}^{-(1-\frac{2}{p})}\big\|_{L_s^\frac{p^2}{(p-1)(p-2)}} \bigg\|\jb{sN}^{1- \frac{2}{p}}\big\|u_{N_1} u_{N_2} u_{N_3} u_{N_4} u_{N_5}\big\|_{L_x^1}\bigg\|_{L_s^{\frac{q}{3}}([0,T_j))} \\
        & \sim \sum_{\substack{N_1\gtrsim N \\ N_5 \leq \dots \leq N_1}}N^{\frac{2}{p} - \frac{2}{p^2}}\bigg\|\jb{sN}^{1- \frac{2}{p}}\big\|u_{N_1} u_{N_2} u_{N_3} u_{N_4} u_{N_5}\big\|_{L_x^1}\bigg\|_{L_s^{\frac{q}{3}}([0,T_j))}.
    \end{align*}
    Bounding $\jb{sN} \lesssim \jb{sN_1}$ and then summing over $\{N : N \lesssim N_1\}$, we find that
    \begin{equation*}\begin{split}
        \big\|\text{NL}\|_{X(T_j)} & \lesssim \sum_{N_5 \leq \dots \leq N_1}N_1^{\frac{2}{p} - \frac{2}{p^2}}\Big\|\jb{sN_1}^{1- \frac{2}{p}}\big\|u_{N_1} u_{N_2} u_{N_3} u_{N_4} u_{N_5}\big\|_{L_x^1}\Big\|_{L_s^{\frac{q}{3}}([0,T_j))} \\
        & = \sum_{N_5 \leq \dots \leq N_1}N_1^{1 - \frac{2}{p^2}}\bigg\|\tfrac{\jb{sN_1}^{1- \frac{2}{p}}}{N_1^{1-\frac{2}{p}}}\big\|u_{N_1} u_{N_2} u_{N_3} u_{N_4} u_{N_5}\big\|_{L_x^1}\bigg\|_{L_s^{\frac{q}{3}}([0,T_j))}.
    \end{split}\end{equation*}
    As $N_1 \geq N_2$, this implies that
    \begin{align*}
        \big\|\text{NL}\|_{X(T_j)} & \lesssim \sum_{N_5 \leq \dots \leq N_1}N_1^{1 - \frac{2}{p^2}}\bigg\|\tfrac{\jb{sN_2}^{1- \frac{2}{p}}}{N_2^{1-\frac{2}{p}}}\|u_{N_2}\|_{L_x^p}\big\|u_{N_1} u_{N_3} u_{N_4} u_{N_5}\big\|_{L_x^{p'}}\bigg\|_{L_s^{\frac{q}{3}}([0,T_j))}.
    \end{align*}
    
    
    We now decompose $[0,T_j)$ into $[0,T_{j-1}) \cup I_j$. We apply H\"older's inequality to the sum over $\{N_2 \in 2^\Z : N_3 \leq N_2 \leq N_1\}$ in order to place $u_{N_2}(s)$ into the bootstrap norm $X(T_j)$ for $s \in I_j$ and into the bootstrap norm $X(T_{j-1})$ for $s \in [0,T_{j-1})$. We may then estimate
    \begin{equation}\begin{split}\label{proof/NL}
        \big\|\text{NL}\|_{X(T_j)} 
        & \lesssim \|u\|_{X(T_j)}\sum_{N_5 \leq N_4 \leq N_3 \leq N_1}\log\big(\tfrac{N_1}{N_3}\big)^{\frac{1}{2}}N_1^{1 - \frac{2}{p^2}}\big\|u_{N_1} u_{N_3} u_{N_4} u_{N_5}\big\|_{L_s^{\frac{q}{3}}L_x^{p'}(I_j)} \\
        & \hspace{11pt} + \|u\|_{X(T_{j-1})}\sum_{N_5 \leq N_4 \leq N_3 \leq N_1}\log\big(\tfrac{N_1}{N_3}\big)^{\frac{1}{2}}N_1^{1 - \frac{2}{p^2}}\big\|u_{N_1} u_{N_3} u_{N_4} u_{N_5}\big\|_{L_s^{\frac{q}{3}}L_x^{p'}} \\
        & = \|u\|_{X(T_j)} \rn{1} + \|u\|_{X(T_{j-1})}\rn{2}.
    \end{split}\end{equation}
    
    We first focus our attention on $\rn{1}$. The estimates for $\rn{2}$ will then follow a symmetric argument. By H\"older's inequality, we may estimate
    \begin{align*}
        \rn{1} & \lesssim \sum_{N_5 \leq N_4 \leq N_3 \leq N_1}\log\big(\tfrac{N_1}{N_3}\big)^{\frac{1}{2}}N_1^{-\frac{2}{p^2}}\|N_1 u_{N_1}\|_{L_s^\infty L_x^2}\prod_{i=3}^5\big\|u_{N_i}\big\|_{L_s^{q}L_x^\frac{6p}{p-2}(I_j)}.
    \end{align*}
    Applying the Bernstein inequality \eqref{intro/bernstein} and then summing over $\{N_4, N_5 : N_5 \leq N_4 \leq N_3$, we find that
    \begin{equation}\begin{split}\label{proof/infty failure 1}
        \rn{1} 
        & \lesssim \sum_{N_5 \leq N_4 \leq N_3 \leq N_1}\log\big(\tfrac{N_1}{N_3}\big)^{\frac{1}{2}}N_1^{-\frac{2}{p^2}}\|N_1 u_{N_1}\|_{L_s^\infty L_x^2}\prod_{i=3}^5N_i^{\frac{2}{3p^2}}\big\|u_{N_i}\big\|_{L_s^{q}L_x^r(I_j)} \\
        & \lesssim \big\|u\big\|^2_{\ell_N^\infty L_s^{q}L_x^r(I_j)}\sum_{N_3 \leq N_1}\log\big(\tfrac{N_1}{N_3}\big)^{\frac{1}{2}}N_3^{\frac{2}{p^2}}N_1^{-\frac{2}{p^2}}\|N_1 u_{N_1}\|_{L_s^\infty L_x^2}\big\|u_{N_3}\big\|_{L_s^{q}L_x^r(I_j)}.
    \end{split}\end{equation}
    Because $p < \infty$, Schur's test and standard arguments then imply that
    \begin{equation}\begin{split}\label{proof/I}
        \rn{1} 
        & \;\lesssim\; \big\|u\big\|^2_{\ell_N^\infty L_s^{q}L_x^r(I_j)}\||\nabla| u\|_{\ell_N^2 L_s^\infty L_x^2}\big\|u\big\|_{\ell_N^2L_s^{q}L_x^r(I_j)}
        \;\lesssim\; C\big(\|(u_0,u_1)\|_{\dot{H}^1 \times L^2}\big)\eta^4.
    \end{split}\end{equation}

    Repeating the above analysis for $\rn{2}$, we may similarly estimate
    \begin{align}\label{proof/II}
        \rn{2}
        & \;\lesssim\; \|u\|^2_{\ell_N^\infty L_s^{q}L_x^r}\big\||\nabla| u\big\|_{\ell_N^2 L_s^\infty L_x^2}\|u\|_{\ell_N^2L_s^{q}L_x^r}
        \;\lesssim\; C\big(\|(u_0,u_1)\|_{\dot{H}^1 \times L^2}\big).
    \end{align}
    Combining the estimates \eqref{proof/NL}, \eqref{proof/I}, and \eqref{proof/II} with the linear estimate \eqref{proof/linear} then yields the bootstrap statement \eqref{proof/bootstrap}. Along with earlier considerations, this concludes the proof of Theorem \ref{theorem}.
\end{proof}

\section{Proof of Theorem \ref{edge theorem}}\label{edge}
    We now prove Theorem \ref{edge theorem}. This proof will parallel the proof of Theorem \ref{theorem} but will require a more careful paraproduct decomposition.

\begin{proof}[Proof of Theorem \ref{edge theorem}]
    It suffices to work with $t > 0$ as $t < 0$ will follow from time-reversal symmetry. By the density of Schwartz functions in $(\dot{B}^{2}_{1,1} \times \dot{B}^1_{1,1}) \cap (\dot{B}^1_{2,1} \times \dot{B}^0_{2,1})$, it suffices to consider Schwartz solutions of \eqref{NLW}.

    We seek to prove the stronger decay result:
    \begin{equation}\label{edge/stronger}
        \Big\|\tfrac{\jb{tN}}{N}u_N\Big\|_{\ell_N^1 L_t^\infty L_x^p([0,\infty))} \leq C\big(\|(u_0,u_1)\|_{\dot{B}^1_{2,1} \times \dot{B}^0_{2,1}}\big)\Big[\|u_0\|_{\dot{B}^2_{1,1}} + \|u_1\|_{\dot{B}^1_{1,1}}\Big].
    \end{equation}
    Note that the sum over frequencies is $\ell^1_N$ as opposed to $\ell_N^2$ in \eqref{proof/stronger}.
    For $0 < T \leq \infty$, we then define the norm
    \begin{equation*}
        \|u\|_{X(T)} = \Big\|\tfrac{\jb{tN}}{N}u_N\Big\|_{\ell_N^2 L_t^\infty L_x^p([0,T))}.
    \end{equation*}
    It then suffices to show
    \begin{equation}\label{edge/conclusion}
        \|u\|_{X(\infty)} \leq C\big(\|(u_0,u_1)\|_{\dot{B}^1_{2,1} \times \dot{B}^0_{2,1}}\big)\big[\|u_0\|_{\dot{B}^2_{1,1}} + \|u_1\|_{\dot{B}^1_{1,1}}\big],
    \end{equation}
    for which we proceed with a bootstrap argument.

    Let $\eta > 0$ denote a small parameter to be chosen later, depending only on universal constants and $\|(u_0,u_1)\|_{\dot{B}^1_{2,1} \times \dot{B}^0_{2,1}}$. Proposition \ref{sb/sb} implies that we may decompose $[0,\infty)$ into $J = J(\|(u_0,u_1)\|_{\dot{B}^1_{2,1} \times \dot{B}^0_{2,1}},\eta)$ many intervals $I_j = [T_{j-1},T_j)$ on which
    \begin{equation}\label{edge/smallness}
        \big\| |\nabla|^{\frac{1}{2}} u\big\|_{\ell_N^1 L^4_{t,x}(I_j)} + \big\|u\big\|_{\ell_N^1 L^8_{t,x}(I_j)} + \big\|u\big\|_{\ell_N^1 L_t^4 L_x^{12}(I_j)} < \eta.
    \end{equation}
    
    For all $j = 1,\dots, J$, we aim to show the bootstrap statement:
    \begin{equation}\begin{split}\label{edge/bootstrap}
        \|u\|_{X(T_j)} & \lesssim \|u_0\|_{\dot{B}^2_{1,1}} + \|u_1\|_{\dot{B}^1_{1,1}} + C\big(\|(u_0,u_1)\|_{\dot{B}^1_{2,1} \times \dot{B}^0_{2,1}}\big)\Big[\|u\|_{X(T_{j-1})} + \eta^3 \|u\|_{X(T_j)}\Big].
    \end{split}\end{equation}
    Choosing $\eta > 0$ sufficiently small based on the constants in \eqref{edge/bootstrap}, we could then iterate over the norms $X(T_j)$ for $j = 1,\dots, J(\|(u_0,u_1)\|_{\dot{B}^1_{2,1} \times \dot{B}^0_{2,1}})$ to yield \eqref{edge/conclusion} and conclude the proof of Theorem \ref{theorem}.
    
    We therefore focus on \eqref{edge/bootstrap}. As before, we recall the Duhamel formula \eqref{proof/duhamel} and note that the contribution of the linear term to $\|u(t)\|_{X(T_j)}$ is acceptable:
    \begin{equation}\label{edge/linear}
        \big\|\cos(t|\nabla|) u_0 + \tfrac{\sin(t|\nabla|)}{|\nabla|} u_1\big\|_{X(T_j)} \lesssim \|u_0\|_{\dot{B}^2_{1,1}} + \|u_1\|_{\dot{B}^1_{1,1}}.
    \end{equation}
    We thus focus our attention on the nonlinear correction $\text{NL}$.

    We decompose $[0,t)$ into $[0,t/2) \cup [t/2, t)$. For the early time interval $[0,t/2)$, we apply the simple paraproduct decomposition \eqref{intro/paraproduct/simple} and then the linear dispersive decay \eqref{intro/linear dispersive decay}. For the late time interval $[t/2,t)$, we apply the full paraproduct decomposition \eqref{intro/paraproduct/full} and consider the two terms separately.
    For $N_1 \gtrsim N$, we apply \eqref{intro/linear dispersive decay} directly.
    For $N_1 \sim N$, we first apply the Bernstein inequality \eqref{intro/bernstein} and then apply the conservation of mass. Estimating in this manner, we may bound the nonlinear correction $\|\text{NL}\|_{X(T_j)}$ by three terms,
    \begin{align*}
        \|\text{NL}\|_{X(T_j)} 
        & \lesssim \sum_{\substack{N_1 \gtrsim N \\ N_5 \leq \dots \leq N_1}} N \sup_{t \in [0,T_j)}\jb{tN} \int_0^{t/2} \jb{(t-s)N}^{-1} \big\|u_{N_1}u_{N_2}u_{N_3}u_{N_4}u_{N_5}(s)\big\|_{L_x^1}ds \\
        & \hspace{11pt} + \sum_{\substack{N_1 \sim N_2 \gtrsim N \\ N_5 \leq \dots \leq N_1}} N \sup_{t \in [0,T_j)}\jb{tN} \int_{t/2}^t \jb{(t-s)N}^{-1} \big\|u_{N_1}u_{N_2}u_{N_3}u_{N_4}u_{N_5}(s)\big\|_{L_x^1}ds \\
        & \hspace{11pt} + \sum_{\substack{N_1 \sim N \\ N_5 \leq \dots \leq N_1}} N^{-\frac{1}{2}} \sup_{t \in [0,T_j)}\jb{tN} \int_{t/2}^t \big\|u_{N_1}u_{N_2}u_{N_3}u_{N_4}u_{N_5}(s)\big\|_{L_x^2}ds \\
        & = \rn{1} + \rn{2} + \rn{3},
    \end{align*}
    which we will treat separately.

    We consider term $\rn{1}$ first. As $s \in [0,t/2)$, we note that $|t-s| \sim |t|$. Summing over $\{N : N \sim N_1\}$, then applying H\"older's inequality, and then noting that $N_5 \leq N_4$, we may estimate $\rn{1}$ as 
    \begin{align*}
        \rn{1} 
        & \lesssim \sum_{N_5 \leq \dots \leq N_1} N_1 \sup_{t \in [0,T_j)} \int_0^{t/2} \big\|u_{N_1}u_{N_2}u_{N_3}u_{N_4}u_{N_5}(s)\big\|_{L_x^1}ds \\
        & \lesssim \||\nabla|u\|_{\ell_N^1 L_s^\infty L_x^2} \sum_{N_5 \leq \dots \leq N_2}\big\|u_{N_2}u_{N_3}u_{N_4}u_{N_5}(s)\big\|_{L^1_sL_x^2([0,T_j))} \\
        & \lesssim \||\nabla| u\|_{\ell_N^1 L_s^\infty L_x^2} \sum_{N_5 \leq \dots \leq N_2}\Big\|\tfrac{\jb{sN_5}}{N_5} \tfrac{N_4}{\jb{s N_4}}u_{N_2}u_{N_3}u_{N_4}u_{N_5}(s)\Big\|_{L^1_sL_x^2([0,T_j))}.
    \end{align*}
    We now decompose $[0,T_j)$ into $[0,T_{j-1}) \cup I_j$. For $s \in [0,T_{j-1})$, we place $u_{N_5}$ into the bootstrap norm $X(T_{j-1})$. For $s \in I_j$, we instead place $u_{N_5}$ into the bootstrap norm $X(T_j)$. H\"older's inequality then implies
    \begin{align*}
        \rn{1} 
        & \lesssim \|u\|_{\ell_N^1 L_s^\infty L_x^2} \|u\|_{X(T_j)}\sum_{N_4 \leq N_3 \leq N_2}\big\|\tfrac{N_4}{\jb{s N_4}}\big\|_{L_s^2}\big\|u_{N_2}u_{N_3}u_{N_4}(s)\big\|_{L^2_{s,x}(I_j)} \\ 
        & \hspace{11pt} + \|u\|_{\ell_N^1 L_s^\infty L_x^2} \|u\|_{X(T_{j-1})}\sum_{N_4 \leq N_3 \leq N_2}\big\|\tfrac{N_4}{\jb{s N_4}}\big\|_{L_s^2}\big\|u_{N_2}u_{N_3}u_{N_4}(s)\big\|_{L^2_{s,x}} \\
        & \lesssim \|u\|_{\ell_N^1 L_s^\infty L_x^2} \|u\|_{X(T_j)}\sum_{N_4 \leq N_3 \leq N_2}N_4^{\frac{1}{2}}\big\|u_{N_2}u_{N_3}u_{N_4}(s)\big\|_{L^2_{s,x}(I_j)} \\ 
        & \hspace{11pt} + \|u\|_{\ell_N^1 L_s^\infty L_x^2} \|u\|_{X(T_{j-1})}\sum_{N_4 \leq N_3 \leq N_2}N_4^{\frac{1}{2}}\big\|u_{N_2}u_{N_3}u_{N_4}(s)\big\|_{L^2_{s,x}} \\
        & \lesssim \|u\|_{\ell_N^1 L_s^\infty L_x^2} \||\nabla|^{1/2}u\|_{\ell_N^1 L^4_{s,x}(I_j)}\|u\|^2_{\ell_N^1 L^8_{s,x}(I_j)}\|u\|_{X(T_j)} \\
        & \hspace{11pt} + \|u\|_{\ell_N^1 L_s^\infty L_x^2} \||\nabla|^{1/2}u\|_{\ell_N^1 L^4_{s,x}}\|u\|^2_{\ell_N^1 L^8_{s,x}}\|u\|_{X(T_{j-1})}.
    \end{align*}
    Proposition \ref{sb/sb} and the smallness condition \eqref{edge/smallness} then imply that
    \begin{equation}\label{edge/I}
        \rn{1} \lesssim C\big(\|(u_0,u_1)\|_{\dot{B}^1_{2,1} \times \dot{B}^0_{2,1}}\big)\Big[\eta^3\|u\|_{X(T_j)} + \|u\|_{X(T_{j-1})}\Big]
    \end{equation}
    which is acceptable for the bootstrap statement \eqref{edge/bootstrap}.

    We now consider term $\rn{2}$. We note that for $s \in [t/2,t)$, $|s| \sim |t|$. H\"older's inequality then implies that
    \begin{align*}
        \rn{2} & \sim \sum_{\substack{N_1 \sim N_2 \gtrsim N \\ N_5 \leq \dots \leq N_1}} N \sup_{t \in [0,T_j)}\int_{t/2}^t \jb{(t-s)N}^{-1} \jb{sN}\big\|u_{N_1}u_{N_2}u_{N_3}u_{N_4}u_{N_5}(s)\big\|_{L_x^1}ds \\
        & \lesssim \sum_{\substack{N_1 \sim N_2 \gtrsim N \\ N_5 \leq \dots \leq N_1}} N^\frac{1}{2} \Big\|\jb{sN}\big\|u_{N_1}u_{N_2}u_{N_3}u_{N_4}u_{N_5}(s)\big\|_{L_x^1}\Big\|_{L_s^2([0,T_j))}.
    \end{align*}
    We bound $\jb{sN} \lesssim \jb{sN_1}$ and then sum over $\{N : N \lesssim N_1\}$. Noting that $N_1 \sim N_2$ and that $N_5 \leq N_1$, we may then estimate
    \begin{align*}
        \rn{2} & \lesssim \sum_{\substack{N_1 \sim N_2\\ N_5 \leq \dots \leq N_1}} N_1^\frac{3}{2} \Big\|\tfrac{\jb{sN_1}}{N_1}\big\|u_{N_1}u_{N_2}u_{N_3}u_{N_4}u_{N_5}(s)\big\|_{L_x^1}\Big\|_{L_s^2([0,T_j))} \\
        & \lesssim \sum_{\substack{N_1 \sim N_2\\ N_5 \leq \dots \leq N_1}} N_1N_2^\frac{1}{2} \Big\|\tfrac{\jb{sN_5}}{N_5}\big\|u_{N_1}u_{N_2}u_{N_3}u_{N_4}u_{N_5}(s)\big\|_{L_x^1}\Big\|_{L_s^2([0,T_j))}.
    \end{align*}
    
    We now decompose $[0,T_j)$ into $[0,T_{j-1})\cup I_j$. Applying H\"older's inequality, we may place $u_{N_5}$ into the bootstrap norm $X(T_{j-1})$ for $s \in [0,T_{j-1})$ and into $X(T_j)$ for $s \in I_j$. Proposition \ref{sb/sb} and \eqref{edge/smallness} then imply that
    \begin{equation}\begin{split}\label{edge/II}
        \rn{2} & \lesssim \|u\|_{X(T_{j-1})}\||\nabla| u\|_{\ell_N^1 L_s^\infty L_x^2} \||\nabla|^{\frac{1}{2}}u\|_{\ell_N^1 L^4_{s,x}} \|u\|^2_{\ell_N^1 L^8_{s,x}}\\
        & \hspace{11pt} + \|u\|_{X(T_{j})}\||\nabla| u\|_{\ell_N^1 L_s^\infty L_x^2} \||\nabla|^{\frac{1}{2}}u\|_{\ell_N^1 L^4_{s,x}(I_j)} \|u\|^2_{\ell_N^1 L^8_{s,x}(I_j)} \\
        & \lesssim C\big(\|(u_0,u_1)\|_{\dot{B}^1_{2,1} \times \dot{B}^0_{2,1}}\big)\Big[\eta^3\|u\|_{X(T_j)} + \|u\|_{X(T_{j-1})}\Big],
    \end{split}\end{equation}
    which is acceptable for \eqref{edge/bootstrap}.

    We finally consider $\rn{3}$. Summing over $\{N : N \sim N_1\}$ and noting that $N_1 \geq N_5$, we may estimate
    \begin{align*}
        \rn{3} & \sim \sum_{\substack{N_1 \sim N \\ N_5 \leq \dots \leq N_1}} N_1^{\frac{1}{2}}\sup_{t \in [0,T_j)}\int_{t/2}^t \tfrac{\jb{sN_1}}{N_1} \big\|u_{N_1}u_{N_2}u_{N_3}u_{N_4}u_{N_5}(s)\big\|_{L_x^2}ds \\
        & \lesssim \sum_{\substack{N_1 \sim N \\ N_5 \leq \dots \leq N_1}} N_1^{\frac{1}{2}}\Big\|\tfrac{\jb{sN_5}}{N_5}\|u_{N_5}(s)\|_{L_x^\infty}\big\|u_{N_1}u_{N_2}u_{N_3}u_{N_4}(s)\big\|_{L_x^2}\Big\|_{L^1_s([0,T_j))}.
    \end{align*}
    We now decompose $[0,T_j) = [0,T_{j-1})\cup I_j$. By H\"older's inequality, we may place $u_{N_5}$ into the bootstrap norm $X(T_{j-1})$ for $s \in [0,T_{j-1})$ and into $X(T_j)$ for $s \in I_j$. We then find that
    \begin{equation}\begin{split}\label{edge/III}
        \rn{3} 
        & \lesssim \|u\|_{X(T_{j-1})}\||\nabla|^{\frac{1}{2}} u\|_{\ell_N^1 L^4_{s,x}}\|u\|^3_{\ell_N^1 L_s^4 L_x^{12}} + \|u\|_{X(T_j)}\||\nabla|^{\frac{1}{2}} u\|_{\ell_N^1 L^4_{s,x}(I_j)}\|u\|^3_{\ell_N^1 L_s^4 L_x^{12}(I_j)} \\
        & \lesssim C\big(\|(u_0,u_1)\|_{\dot{B}^1_{2,1} \times \dot{B}^0_{2,1}}\big)\Big[\eta^3\|u\|_{X(T_j)} + \|u\|_{X(T_{j-1})}\Big],
    \end{split}\end{equation}
    which is acceptable for \eqref{edge/bootstrap}.

    Combining the estimates \eqref{edge/I}, \eqref{edge/II}, and \eqref{edge/III} with the linear estimate \eqref{edge/linear} then implies the bootstrap statement \eqref{edge/bootstrap}. Along with earlier considerations, this concludes the proof of Theorem \ref{edge}.
\end{proof}

\bibliographystyle{abbrv}
\bibliography{references}

\begin{thebibliography}{10}

\bibitem{besov-notation}
H.~Bahouri, J.-Y. Chemin, and R.~Danchin.
\newblock {\em Fourier Analysis and Nonlinear Partial Differential Equations}.
\newblock Springer Berlin Heidelberg, 2011.

\bibitem{defocusing-Bahouri}
H.~Bahouri and J.~Shatah.
\newblock Decay estimates for the critical semilinear wave equation.
\newblock {\em Ann. Inst. H. Poincar\'e{} C Anal. Non Lin\'eaire}, 15(6):783--789, 1998.

\bibitem{eNLS-wp-d3}
J.~Colliander, M.~Keel, G.~Staffilani, H.~Takaoka, and T.~Tao.
\newblock Global well-posedness and scattering for the energy-critical nonlinear {S}chr\"{o}dinger equation in {$\R^3$}.
\newblock {\em Ann. of Math. (2)}, 167(3):767--865, 2008.

\bibitem{mNLS-dd}
C.~Fan, R.~Killip, M.~Vi\c{s}an, and Z.~Zhao.
\newblock Dispersive decay for the mass-critical nonlinear {S}chr\"odinger equation.
\newblock Preprint \texttt{arXiv:2403.09989}.

\bibitem{cNLS-dd-constant}
C.~Fan, G.~Staffilani, and Z.~Zhao.
\newblock On decaying properties of nonlinear {S}chr\"odinger equations.
\newblock {\em SIAM J. Math. Anal.}, 56(3):3082--3109, 2024.

\bibitem{defocusing-Grillakis}
M.~G. Grillakis.
\newblock Regularity and asymptotic behaviour of the wave equation with a critical nonlinearity.
\newblock {\em Ann. of Math. (2)}, 132(3):485--509, 1990.

\bibitem{defocusing-Grillakis-2}
M.~G. Grillakis.
\newblock Regularity for the wave equation with a critical nonlinearity.
\newblock {\em Communications on pure and applied mathematics}, 45(6):749--774, 1992.

\bibitem{eNLS-dd-H3}
Z.~Guo, C.~Huang, and L.~Song.
\newblock Pointwise decay of solutions to the energy critical nonlinear {S}chr\"odinger equations.
\newblock {\em Journal of Differential Equations}, 366:71--84, 2023.

\bibitem{cubic-1986}
N.~Hayashi and M.~Tsutsumi.
\newblock {$L^\infty(\mathbb{R}^n)$}-decay of classical solutions for nonlinear {S}chr\"odinger equations.
\newblock {\em Proceedings of the Royal Society of Edinburgh: Section A Mathematics}, 104(3–4):309–327, 1986.

\bibitem{focusing-GWP}
C.~E. Kenig and F.~Merle.
\newblock Global well-posedness, scattering and blow-up for the energy-critical focusing non-linear wave equation.
\newblock {\em Acta Math.}, 201(2):147--212, 2008.

\bibitem{we-1980}
S.~Klainerman.
\newblock Global existence for nonlinear wave equations.
\newblock {\em Communications on Pure and Applied Mathematics}, 33(1):43--101, 1980.

\bibitem{we-1985}
S.~Klainerman.
\newblock Uniform decay estimates and the {L}orentz invariance of the classical wave equation.
\newblock {\em Comm. Pure Appl. Math.}, 38(3):321--332, 1985.

\bibitem{general-1983}
S.~Klainerman and G.~Ponce.
\newblock Global, small amplitude solutions to nonlinear evolution equations.
\newblock {\em Comm. Pure Appl. Math.}, 36(1):133--141, 1983.

\bibitem{eNLS-dd}
M.~Kowalski.
\newblock Dispersive decay for the energy-critical nonlinear {S}chr\"odinger equation.
\newblock Preprint \texttt{arXiv:2411.01466}.

\bibitem{cubic-1978}
J.-E. Lin and W.~A. Strauss.
\newblock Decay and scattering of solutions of a nonlinear {S}chr\"odinger equation.
\newblock {\em Journal of Functional Analysis}, 30(2):245--263, 1978.

\bibitem{looi-2}
S.-Z. Looi.
\newblock Pointwise decay for the energy-critical nonlinear wave equation.
\newblock Preprint \texttt{arXiv:2205.13197}.

\bibitem{looi-1}
S.-Z. Looi.
\newblock Pointwise decay for the wave equation on nonstationary spacetimes.
\newblock Preprint \texttt{arXiv:2105.02865}.

\bibitem{eNLS-dd-H1+}
C.~Ma, H.~Wang, X.~Yu, and Z.~Zhao.
\newblock On the decaying property of quintic {NLS} on 3d hyperbolic space.
\newblock {\em Nonlinear Analysis}, 247:113599, 2024.

\bibitem{rel-we-1972}
C.~S. Morawetz and W.~A. Strauss.
\newblock Decay and scattering of solutions of a nonlinear relativistic wave equation.
\newblock {\em Comm. Pure Appl. Math.}, 25:1--31, 1972.

\bibitem{eNLS-wp-d4}
E.~Ryckman and M.~Vi\c{s}an.
\newblock Global well-posedness and scattering for the defocusing energy-critical nonlinear {S}chr\"{o}dinger equation in {$\R^{1+4}$}.
\newblock {\em Amer. J. Math.}, 129(1):1--60, 2007.

\bibitem{segal}
I.~E. Segal.
\newblock Nonlinear wave equations.
\newblock In {\em Nonlinear partial differential operators and quantization procedures ({C}lausthal, 1981)}, volume 1037 of {\em Lecture Notes in Math.}, pages 115--141. Springer, Berlin, 1983.

\bibitem{mgKdV-dd}
M.~Shan.
\newblock Dispersive decay for the mass-critical generalized {K}orteweg-de {V}ries equation and generalized {Z}akharov--{K}uznetsov equations.
\newblock Preprint \texttt{arXiv:2409.05550}.

\bibitem{defocusing-Sha-Struwe}
J.~Shatah and M.~Struwe.
\newblock Regularity results for nonlinear wave equations.
\newblock {\em Annals of Mathematics}, 138(3):503--518, 1993.

\bibitem{defocusing-Sha-Struwe-2}
J.~Shatah and M.~Struwe.
\newblock Well-posedness in the energy space for semilinear wave equations with critical growth.
\newblock {\em International Mathematics Research Notices}, 1994(7):303--309, 1994.

\bibitem{shatah-struwe}
J.~Shatah and M.~Struwe.
\newblock {\em Geometric wave equations}, volume~2 of {\em Courant Lecture Notes in Mathematics}.
\newblock New York University, Courant Institute of Mathematical Sciences, New York; American Mathematical Society, Providence, RI, 1998.

\bibitem{sogge}
C.~D. Sogge.
\newblock {\em Lectures on non-linear wave equations}, volume~2.
\newblock International Press Boston, MA, 1995.

\bibitem{defocusing-Struwe}
M.~Struwe.
\newblock Globally regular solutions to the $ u^5$ {K}lein--{G}ordon equation.
\newblock {\em Annali della Scuola Normale Superiore di Pisa-Classe di Scienze}, 15(3):495--513, 1988.

\bibitem{spacetime-bounds-tao}
T.~Tao.
\newblock Spacetime bounds for the energy-critical nonlinear wave equation in three spatial dimensions.
\newblock {\em Dyn. Partial Differ. Equ.}, 3(2):93--110, 2006.

\end{thebibliography}
\end{document}